\documentclass[11pt,a4paper]{scrartcl}

\makeatletter
\def\@fnsymbol#1{\textsuperscript{\it{\ifcase#1\or a\or b\or c\or d\or e\or
	f\else\@ctrerr\fi}}}
\makeatother

\usepackage{amsmath}
\usepackage{amsthm}
\usepackage{amssymb}
\usepackage{algorithmic}
\usepackage{algorithm}

\usepackage[neverdecrease]{paralist}

\theoremstyle{plain}
\newtheorem{thm}{Theorem}[section]
\newtheorem{lem}[thm]{Lemma}
\newtheorem{prop}[thm]{Proposition}
\newtheorem{cor}[thm]{Corollary}
\theoremstyle{definition}
\newtheorem{defn}[thm]{Definition}
\newtheorem*{remark}{Remark}

\def\sr{sign-re\-vers\-ing}

\def\spp{simple principal pivot}

\def\ole{\overline{e}}
\def\olf{\overline{f}}
\def\olF{\overline{F}}

\def\ulX{\underline{X}}
\def\ulY{\underline{Y}}

\DeclareMathOperator{\sgn}{sgn}
\let\reflection\Re

\def\M{\mathcal{M}}
\def\V{\mathcal{V}}
\def\C{\mathcal{C}}
\def\D{\mathcal{D}}

\def\R{\mathbb{R}}

\def\T{^{\mathrm{T}}}

\def\veref#1{(V\ref{#1})}
\def\ciref#1{(C\ref{#1})}

\let\backslash\setminus

\title{Combinatorial Characterizations of~K-matrices\thanks
	{This research was partially supported by the project
	`A Fresh Look at the Complexity of Pivoting in Linear Complementarity'
	no.~200021-124752~/~1 of the Swiss National Science Foundation.}}

\author{Jan Foniok%
\thanks{ETH Zurich, Institute for Operations Research, 8092 Zurich, Switzerland}
\thanks{foniok@math.ethz.ch}
\and
Komei Fukuda%
\footnotemark[2]
\thanks{ETH Zurich, Institute of Theoretical Computer Science, 8092 Zurich, Switzerland}
\thanks{fukuda@math.ethz.ch}
\and
Lorenz Klaus%
\footnotemark[2]
\thanks{lklaus@math.ethz.ch}}

\begin{document}

\maketitle

\begin{abstract}
We present a number of combinatorial characterizations of
K-matrices.  This extends a theorem of Fiedler and Pt\'ak
on linear-algebraic characterizations of K-matrices to
the setting of oriented matroids.  Our proof is
elementary and simplifies the original proof substantially
by exploiting the duality of oriented matroids.
As an application, we show that a simple
principal pivot method applied to the linear complementarity
problems with K-matrices converges very quickly,
by a purely combinatorial argument.


\end{abstract}

\section{Introduction}

A matrix $M\in\R^{n\times n}$
is a \emph{P-matrix} if all its principal minors (determinants of
principal submatrices) are positive; it is a \emph{Z-matrix} if all
its off-diagonal elements are non-positive; and it is a \emph{K-matrix}
if it is both a P-matrix and a Z-matrix.

Z- and K-matrices often occur in a wide variety of areas such as
input--output production and growth models in economics, finite
difference methods for partial differential equations, Markov processes
in probability and statistics, and linear complementarity problems in
operations research~\cite{BerPle:NNM}.

In 1962, Fiedler and Pt\'ak~\cite{FiePta:On-matrices-with} listed thirteen
equivalent conditions for a Z-matrix to be a K-matrix. Some of them
concern the sign structure of vectors:

\begin{thm}[Fiedler--Pt\'ak~\cite{FiePta:On-matrices-with}]
\label{thm:FP}
Let $M$ be a Z-matrix. Then the following conditions are equivalent:
\begin{compactenum}[\rm(a)]
\item There exists a vector $x\ge0$ such that $Mx>0$;
\item there exists a vector $x>0$ such that $Mx>0$;
\item the inverse~$M^{-1}$ exists and $M^{-1}\ge 0$;
\item for each vector $x\ne0$ there exists an index~$k$ such that $x_ky_k>0$ for
$y=Mx$;
\item all principal minors of~$M$ are positive
(that is, $M$~is a P-matrix, and thus a \mbox{K-matrix}).
\end{compactenum}
\end{thm}

Our interest in K-matrices originates in the \emph{linear
complementarity problem (LCP)}, which is for a given matrix $M \in
\mathbb{R}^{n \times n}$ and a given vector $q \in \mathbb{R}^{n}$
to find two vectors $w$ and $z$ in $\mathbb{R}^{n}$ so that
\begin{equation}
\label{eq:LCP}
\begin{aligned}
 w-Mz&=q, \\ 
 w,z &\geq 0,\\
 w\T z&=0.
\end{aligned}
\end{equation}

In general, the problem to decide whether a LCP has a solution
is NP-com\-plete~\cite{Chung:Hardness,KojMegNom:A-unified}.  If the matrix~$M$
is a P-matrix, however, a unique solution to the LCP always
exists~\cite{STW}. Nevertheless, no polynomial-time algorithm to find
it is known, nor are hardness results for this intriguing class of LCPs.
It is unlikely to be NP-hard, because that would imply that
NP${}={}$co-NP~\cite{Meg:A-Note-on-the-Complexity}.  Recognizing whether a
matrix is a P-matrix is co-NP-complete~\cite{Cox:The-P-matrix}.  For some
recent results, see also~\cite{MorNam:sandwiches}.

If the matrix~$M$ is a Z-matrix, a polynomial-time (pivoting) algorithm
exists~\cite{Cha:A-special} (see also \cite[sect.~8.1]{Mur:Linear})
that finds the solution or concludes that no
solution exists. Interestingly, LCPs over this simple class of matrices
have many practical applications (pricing of American options, portfolio
selection problems, resource allocation problems).

A frequently considered class of algorithms to solve LCPs is the class of
\emph{simple principal pivoting methods} (see Section~\ref{sec:algo} or
\cite[Sect.~4.2]{CotPanSto:LCP}). We speak about a \emph{class} of algorithms
because the concrete algorithm depends on a chosen \emph{pivot
rule}. It has recently been proved in~\cite{FonFukGar:Pivoting} that
a simple principal pivoting method with \emph{any} pivot rule takes
at most a number of pivot steps linear in~$n$ to solve a LCP with a
K-matrix~$M$.

The study of pivoting methods and pivot rules has led to the devising
of \emph{combinatorial abstractions} of LCPs. One such abstraction is
unique-sink orientations of cubes~\cite{StiWat:Digraph-models}; the one
we are concerned with here is one of oriented matroids.

Oriented matroids were pioneered by Bland and Las Vergnas~\cite{bl-om-78}
and Folkman and Lawrence~\cite{fl-om-78}. Todd~\cite{t-com-84} and
Morris~\cite{Mor:Oriented} gave a combinatorial generalization of
LCPs by formulating the complementarity problem of oriented matroids
(OMCP).
Morris and Todd~\cite{MorTod:Symmetry} studied properties
of matroids extending LCPs with symmetric and positive
definite matrices.
Todd~\cite{t-com-84} proposed a generalization of Lemke's
method~\cite{Lemke:Bimatrix} to solve the OMCP. Later Klafszky and
Terlaky~\cite{KlaTer:Some} and Fukuda and Terlaky~\cite{FukTer:Linear}
proposed a generalized criss-cross method; in~\cite{FukTer:Linear}
it is used for a constructive proof of
a duality theorem for OMCPs in sufficient matroids (and hence also for
LCPs with sufficient matrices).
Hereby we revive their approach.

In this paper, we present a combinatorial generalization
(Theorem~\ref{thm:eqK}) of the Fiedler--Pt\'ak Theorem~\ref{thm:FP}. To
this end, we devise oriented-matroid counterparts of the conditions
(a)--(d). If the oriented matroid in question is realizable as the sign
pattern of the null space of a matrix, then our conditions are equivalent
to the conditions on the realizing matrix. In general, however,
our theorem is stronger because it applies also to nonrealizable
oriented matroids.

As a by-product, our proof yields a new, purely combinatorial proof of
Theorem~\ref{thm:FP}. Rather than on algebraic properties, it relies
heavily on oriented matroid duality.

We then use our characterization theorem to show  that an OMCP
on an $n$-dimensional K-matroid (that is, a matroid satisfying
the equivalent conditions of Theorem~\ref{thm:eqK}) is solved by any
pivoting method in at most~$2n$ pivot steps. This implies the result
of~\cite{FonFukGar:Pivoting} mentioned above that any simple principal
pivoting method is fast for K-matrix LCPs.

\section{Oriented matroids}

The theory of oriented matroids provides a natural concept which not only
generalizes combinatorial properties of many geometric configurations
but presents itself in many other areas as well, such as topology and
theoretical chemistry.

\subsection{Definitions and basic properties}

Here we state the definitions and basic properties of oriented matroids
that we need in our exposition. For more on oriented matroids consult,
for instance,~the book~\cite{bk-om}.

Let $E$ be a finite set of size $n$. A \emph{sign vector} on~$E$
is a vector~$X$ in $\left\lbrace +1,0,-1 \right\rbrace^{E}$. Instead
of~$+1$, we write just~$+$; instead of~$-1$, we write just~$-$. We define
$X^{-}=\left\lbrace e \in E : X_{e}=- \right\rbrace$, $X^{\ominus}=\{e\in
E: X_e={-} \text{ or } X_e=0\}$, and the sets $X^{0}$, $X^{\oplus}$
and $X^{+}$ analogously. For any subset $F$ of $E$ we write $X_{F}
\geq 0$ if $F \subseteq X^{\oplus}$, and $X_{F} \leq 0$ if $F \subseteq
X^{\ominus}$; furthermore $X\ge0$ if $X_{E}\ge0$ and $X\le0$ if
$X_{E}\le0$. The \emph{support} of a sign vector~$X$ is $\ulX = X^+\cup
X^-$. The \emph{opposite} of $X$ is the sign vector~$-X$ with $(-X)^{+}=X^{-}$,
$(-X)^{-}=X^{+}$ and $(-X)^{0}=X^{0}$. The \emph{composition} of two
sign vectors $X$ and $Y$ is given by
\[ (X \circ Y)_e = \begin{cases}
X_{e} & \text{if  $X_{e} \neq 0$,} \\
Y_{e} & \text{otherwise.}
\end{cases}\]
The \emph{product} $X \cdot Y$ of two sign vectors is the sign vector given by
\[(X \cdot Y)_{e}=
\begin{cases}
0 & \text{if } X_{e}=0 \text{ or } Y_{e}=0, \\
+ & \text{if } X_{e}=Y_{e} \text{ and } X_{e} \neq 0, \\
- & \text{otherwise.}
\end{cases} \]

\begin{defn} \label{def:vecM}
An \emph{oriented matroid} on $E$ is a pair $\M=(E,\mathcal{V})$, where $\mathcal{V}$ is a set of sign vectors on $E$ satisfying the following axioms:

\begin{compactenum}[(V1)]
\item $0\in \mathcal{V}$. \label{v1}
\item If $X \in \mathcal{V}$, then $-X \in \mathcal{V}$.  \label{v2}
\item If $X,Y \in \mathcal{V}$, then $X \circ Y \in \mathcal{V}$. \label{v3}
\item If $X,Y \in \mathcal{V}$ and $e \in X^+\cap Y^-$,
	then there exists $Z\in\V$ with $Z^+\subseteq X^+\cup Y^+$,
	$Z^-\subseteq X^-\cup Y^-$, 
	$Z_e=0$, and
	$(\ulX\setminus\ulY) \cup (\ulY\setminus\ulX) \cup (X^+\cap Y^+)
	\cup (X^-\cup Y^-) \subseteq \underline{Z}$. \label{v4}
\end{compactenum}
\end{defn}

The axioms \veref{v1} up to \veref{v4} are called \emph{vector axioms};
\veref{v4} is the \emph{vector elimination axiom}.
We say that the sign vector~$Z$ is the result of a vector elimination
of $X$ and~$Y$ at element~$e$.

An important example is a matroid whose vectors are the
sign vectors of elements of a vector subspace of~$\R^n$. If $A$~is an $r\times n$ real
matrix, define
\begin{equation}
\label{eq:realmat}
\V=\{\sgn x : x\in\R^n\text{ and }Ax=0\},
\end{equation}
where $\sgn x =
(\sgn x_1,\dotsc,\sgn x_n)$. Then $\V$~is the vector set of an oriented
matroid on the set $E=\{1,2,\dotsc,n\}$. In this case we speak of
\emph{realizable} oriented matroids.

A \emph{circuit} of~$\M$ is a nonzero vector~$\C\in\V$ such that
there is no nonzero vector~$X\in\V$ satisfying $\underline{X} \subset
\underline{C}$.

\begin{prop}
Let $\M=(E,\V)$ be a matroid and let $\C$ be the collection of all its circuits. Then:
\begin{compactenum}[\rm(C1)]
\item $0\not \in \mathcal{C}$. \label{c1}
\item If $C \in \mathcal{C}$, then $-C \in \mathcal{C}$. \label{c2}
\item For all $C,D \in \mathcal{C}$, if $\underline{C} \subseteq \underline{D}$,
	then $C=D$  or $C=-D$. \label{c3}
\item If $C,D \in \mathcal{C}$, $C \neq -D$  and $e \in C^+\cap D^-$, then
	there is a $Z \in \mathcal{C}$ with
	$Z^{+} \subseteq (C^{+} \cup D^{+})\backslash \left\lbrace  e \right\rbrace$
	and $Z^{-} \subseteq (C^{-} \cup D^{-})\backslash \left\lbrace e \right\rbrace$.
	\label{c4}
\item If $C,D \in \mathcal{C}$, $e \in C^+\cap D^-$  and
	$f \in (C^{+} \backslash D^{-}) \cup (C^{-} \backslash D^{+})$, then
	there is a $Z \in \mathcal{C}$ with
	$Z^{+} \subseteq (C^{+} \cup D^{+})\backslash \left\lbrace  e \right\rbrace$,
	$Z^{-} \subseteq (C^{-} \cup D^{-})\backslash \left\lbrace  e \right\rbrace$
	and $Z_f \ne0$. \label{c4p} 
\item For every vector $X\in\V$ there exist circuits $C^1,C^2,\dotsc,C^k\in\C$ such that
	$X=C^1\circ C^2\circ\dotsb\circ C^k$ and $C^{i}_{e}\cdot C^{j}_{e} \ge 0$
	for all indices $i,j$ and all $e\in\ulX$. \label{c6}
\end{compactenum}

\medskip\noindent
Moreover, if a set $\C$ of sign vectors on~$E$ satisfies
\ciref{c1}--\ciref{c4}, then it is the set of all circuits of a unique
matroid; this matroid's vectors are then all finite compositions of
circuits from~$\C$.
\end{prop}

The property \ciref{c4} is called \emph{weak circuit elimination};
\ciref{c4p} is called \emph{strong circuit elimination}.
In \ciref{c6} we speak about a \emph{conformal decomposition} of a vector into circuits.

A \emph{basis} of an oriented matroid~$\mathcal{M}$ is an
inclusion-maximal set $B \subseteq E$ for which there is no circuit~$C$
with $\underline{C} \subseteq B$. Every basis~$B$ has the same size,
called the \emph{rank} of~$\mathcal{M}$.

\begin{prop}
Let $B$ be a basis of an oriented matroid~$\M$.  For every $e$
in $E \backslash B$ there is a unique circuit $C(B,e)$ such that
$\underline{C(B,e)} \subseteq B \cup \left\lbrace e\right\rbrace$
and $C(B,e)_{e}=+$.
\end{prop}

Such a circuit $C(B,e)$ is called the \emph{fundamental circuit} of~$e$
with respect to~$B$.

\bigskip 

Two sign vectors $X$ and $Y$ are \emph{orthogonal} if the set $\left\lbrace
X_{e} \cdot Y_{e} : e \in E \right\rbrace$ either equals $\left\lbrace
0\right\rbrace$ or contains both $+$ and $-$. We then write $X \perp Y$.

\begin{prop}
For every oriented matroid $\mathcal{M}=(E,\mathcal{V})$ of rank $n$
there is a unique oriented matroid $\mathcal{M}^{*}=(E,\mathcal{V}^{*})$
of rank $\left| E \right| - n$ given by
 $$ \mathcal{V}^{*} = \left\lbrace Y \subseteq \left\lbrace -,0,+ \right\rbrace^{E} :
 X \perp Y \text{ for every } X \in \mathcal{V} \right\rbrace.$$
\end{prop}

This $\mathcal{M}^{*}$ is called the \emph{dual} of $\M$. Note that
$(\mathcal{M}^{*})^{*}=\mathcal{M}$. The circuits of~$\M^\ast$ are called
the \emph{cocircuits} of~$\M$ and the vectors of~$\M^\ast$ are called
the \emph{covectors} of~$\M$.
The covectors of a realizable matroid given by~\eqref{eq:realmat} are
sign vectors of the elements of the row space of the matrix~$A$.

We conclude this short overview by introducing the concept of matroid
minors and extensions. For any $F \subseteq E$, the vector $X \backslash
F$ denotes the subvector $(X_{e}: e \in E \backslash F)$ of~$X$. Then let
$$ \mathcal{V} \backslash F= \left\lbrace X \backslash F : X \in \mathcal{V}
\text{ and } X_{f}=0 \text{ for all } f \in F \right\rbrace $$
be the \emph{deletion} and
$$ \mathcal{V} \slash F= \left\lbrace X \backslash F : X \in \mathcal{V} \right\rbrace $$
the \emph{contraction} of the vectors in $\mathcal{V}$ by the elements of
$F$. It is easy to check that the pairs $\mathcal{M} \backslash F =(E
\backslash F,\mathcal{V} \backslash F)$ and $\mathcal{M} \slash F=
(E \backslash F,\mathcal{V} \slash F)$ are oriented matroids. For any
disjoint $F,G \subseteq E$ we call the oriented matroid $(\mathcal{M}
\backslash F) \slash G$ a \emph{minor} of~$\mathcal{M}$.

Note that $\M\setminus\{e,e'\}=(\M\setminus\{e\})\setminus\{e'\}$, 
$\M/\{e,e'\}=(\M/\{e\})/\{e'\}$ and
$({\M\setminus\{e\}})\slash\{e'\}=(\M/\{e'\})\setminus\{e\}$, and so
deletions and contractions can be performed element by element
in any order, with the same result.

\begin{defn}
A matroid $\hat{\mathcal{M}}=(E \cup \left\lbrace q \right\rbrace,
\hat{\mathcal{V}})$ with $q \not \in E$ is a \emph{one-point extension}
of~$\mathcal{M}$ if $\hat{\mathcal{M}} \backslash \left\lbrace
q\right\rbrace = \mathcal{M}$ and there is a vector~$X$ of~$\hat\M$
with $X_q\ne0$.
\end{defn}

\subsection{Complementarity in oriented matroids}

In the rest of the paper, we are considering oriented matroids endowed
with a special structure. The set of elements $E_{2n}$ is a $2n$-element
set with a fixed partition $E_{2n}=S\cup T$ into two $n$-element sets
and a mapping $e\mapsto \ole$ from $E_{2n}$ to~$E_{2n}$ which is an
involution (that is, $\overline{\ole}=e$ for every $e\in E_{2n}$) and for
every $e\in S$ we have $\ole\in T$. Note that this mapping constitutes
a bijection between $S$ and~$T$.

The element $\ole$ is called the \emph{complement} of $e$. For a subset
$F$ of $E_{2n}$ let $\overline{F} =\left\lbrace \overline{e} : e \in F
\right\rbrace$. A subset~$F$ of~$E_{2n}$ is called \emph{complementary}
if $F\cap \overline{F}=\emptyset$.

The matroids we are working with are of the kind $\M=(E_{2n},\V)$,
where the set $S\subseteq E_{2n}$ is a basis of~$\M$.  In
addition, we study their one-point extensions $\hat\M=(\hat
E_{2n},\hat\V)$, where $\hat E_{2n} = E_{2n}\cup\{q\}$ for some element
$q\notin E_{2n}$.
Sometimes we make the canonical choice $E_{2n}=\{1,\dotsc,2n\}$ with
$S=\{1,\dotsc,n\}$ where the complement of an $i \in S$ is the element
$i+n$.

\begin{defn}
The \emph{oriented matroid complementarity problem (OMCP)} is to find
a vector $X$ of an oriented matroid $\hat\M$ so that
\begin{subequations}
\begin{gather}
 X \in \hat{\mathcal{V}}, \label{VinV}\\ 
 X \geq 0, \text{ } X_{q}=+, \label{VPOS} \\
 X_{e} \cdot X_{\ole}=0  \text{\quad for every } e \in E_{2n}, \label{V2COMP}
\end{gather}
\end{subequations}
or to report that no such vector exists.

A vector~$X$ which satisfies~\eqref{VPOS} is called \emph{feasible}, one
which satisfies~\eqref{V2COMP} is called \emph{complementary}. Note
that a vector is complementary if and only if its support is
a complementary set. If an $X \in \hat{\mathcal{V}}$ satisfies
\eqref{VPOS} and~\eqref{V2COMP}, then $X$~is a \emph{solution} to the
OMCP($\mathcal{\hat{M}}$).
\end{defn}

Now we show how LCPs are special cases of OMCPs. Finding a solution to the
LCP~\eqref{eq:LCP} is equivalent to finding an element~$x$ of
$$V=\Bigl\{ x \in \mathbb{R}^{2n+1} :
\begin{bmatrix} I_{n} & -M & -q \end{bmatrix}x=0 \Bigr\}$$
such that\begin{equation}
\label{eq:LCPalt}
\begin{gathered}
 x \geq 0, \text{ } x_{2n+1}=1,\\
 x_{i} \cdot x_{i+n}=0 \text{\quad for every } i \in \left[ n\right].
\end{gathered}
\end{equation}
We set $\hat\V=\{\sgn x: x\in V\}$ and consider the OMCP for the matroid
$\hat\M=(\hat E_{2n},\hat\V)$. Clearly, if the OMCP has no solution, then
$V$~contains no vector~$x$ satisfying~\eqref{eq:LCPalt}. If on the other
hand there is a solution~$X$ satisfying \eqref{VinV}--\eqref{V2COMP},
then the solution to the LCP can be obtained by solving the system of
linear equations
\begin{alignat*}{2}
\begin{bmatrix} I_{n} & -M & -q \end{bmatrix}x &= 0, \\
x_i &= 0 &\quad\text{whenever $X_i=0$},\\
x_{2n+1} &= 1.
\end{alignat*}
This correspondence motivates the following definition.

\begin{defn} \label{def:realM}
An oriented matroid $\M=(E_{2n},\V)$ is \emph{LCP-realizable} if there
is a matrix $M \in \mathbb{R}^{n \times n}$ such that
\[\V = 
\Bigl\{ \sgn x : x\in\R^{2n} \text{ and }
\begin{bmatrix} I_{n} & -M  \end{bmatrix}x=0 \Bigr\}.\]
The matrix~$M$ is then a \emph{realization matrix} of~$\M$.
This is a little nonstandard, because usually the matrix~$A$
from~\eqref{eq:realmat} is called a realization matrix.
The columns of~$I_n$ are indexed by the elements of $S\subset E_{2n}$,
and the columns of~$-M$ are indexed by the elements of $T\subset E_{2n}$
so that if the $k$th column of~$I_n$ is indexed by~$e$, then the $k$th
column of~$-M$ is indexed by~$\ole$.

The extension $\hat\M=(\hat E_{2n},\hat\V)$ is \emph{LCP-realizable} if
there is a matrix $M \in \mathbb{R}^{n \times n}$ and a vector $q\in\R^n$
such that
\[\hat\V = 
\Bigl\{ \sgn x : x\in\R^{2n+1} \text{ and }
\begin{bmatrix} I_{n} & -M & -q \end{bmatrix}x=0 \Bigr\}.\]
\end{defn}

To study the algorithmic complexity of OMCPs, we must specify how the
matroid~$\hat\M$ is made available to the algorithm. We will assume that
it is given by an oracle which, for a basis~$B$ of~$\hat\M$ and a nonbasic
element $e\in\hat E_{2n}\setminus B$, outputs the unique (fundamental)
circuit~$C$ of~$\hat\M$ with support $\underline{C}\subseteq B\cup\{e\}$
such that $X_e={+}$.

In the LCP-realizable case such an oracle can be implemented in
polynomial time; in fact, it consists in solving a system of $O(n)$
linear equations in $2n+1$ variables.  Thus, in the RAM model, the
oracle can be implemented so that it performs arithmetic operations
whose number is bounded by a polynomial in~$n$. Hence our goal is to
develop an algorithm that solves an OMCP using a number of queries to
the oracle that is polynomial in~$n$.

Such an algorithm for the OMCP would obviously provide a strongly
polynomial algorithm for the LCP. Since the LCP is NP-hard in general,
the existence of such an algorithm is unlikely. In Section~\ref{sec:algo}
we do, nevertheless, prove the existence of such an algorithm for a
special class of oriented matroids: K-matroids.


\section{P-matroids}

In this and the following sections, we investigate what properties
of oriented matroids characterize oriented matroids realizable by
special classes of matrices. We start with P-matrices; recall that a
P-matrix is a matrix whose principal minors are all positive.

Several conditions are equivalent to the positivity of principal minors:

\begin{thm}
\label{thm:3.1}
For a matrix $M\in\R^{n\times n}$, the following are equivalent:
\begin{compactenum}[\rm(a)]
\item All principal minors of~$M$ are positive
	(i.e., $M$ is a P-matrix);\label{3.1-a}
\item there is no nonzero vector~$x$ such that $x_k y_k\leq0$
for all $i=1,2,\dotsc,n$, where $y=Mx$;\label{3.1-b}
\item the LCP~\eqref{eq:LCP} with the matrix~$M$ and any right-hand
side~$q$ has exactly one solution.\label{3.1-c}
\end{compactenum}
\end{thm}

The equivalence of (\ref{3.1-a}) and (\ref{3.1-b}) is due to Fiedler
and Pt\'ak~\cite{FiePta:On-matrices-with}. The equivalence of
(\ref{3.1-a}) and (\ref{3.1-c}) was proved independently by Samelson,
Thrall and Wesler~\cite{STW}, Ingleton~\cite{Ing:A-problem}, and
Murty~\cite{Mur:On-the-number}.

The following notions and our definition of a P-matroid are motivated by
the condition~(\ref{3.1-b}) in Theorem~\ref{thm:3.1}. It is much easier
to express in the oriented-matroid language than~(\ref{3.1-a}).

A sign vector $X \in \{-,0,+\}^{E_{2n}}$ is \emph{sign-reversing
(s.r.)}~if $X_{e} \cdot X_{\ole} \leq 0$ for every $e \in S$. If in
addition $\underline{X}=E_{2n}$, the vector is \emph{totally
sign-reversing (t.s.r.)}. Analogously, an $X$ is \emph{sign-preserving
(s.p.)}~if $X_{e} \cdot X_{\ole} \geq 0$ for every $e$, and \emph{totally
sign-preserving (t.s.p.)}~if $\underline{X}=E_{2n}$ as well.

\begin{defn}[Todd~\cite{t-com-84}] \label{def:P}
An oriented matroid $\mathcal{M}$ on $E_{2n}$ is a \emph{P-matroid}
if it has no sign-reversing circuit.
\end{defn}

Note that a P-matroid contains no nonzero \sr\ vectors, because every
vector is the composition of some circuits and composing non-s.r.\
circuits yields non-s.r.\ vectors. Hence, using Theorem~\ref{thm:3.1},
we immediately get:

\begin{prop}
\label{realP}
~
\begin{compactenum}[\rm(i)]
\item If $\M$ is LCP-realizable and there exists a realization matrix~$M$
that is a P-matrix, then $\M$~is a P-matroid.
\item If $\mathcal{M}$ is an LCP-realizable P-matroid, then
every realization matrix~$M$ is a P-matrix.
\end{compactenum}
\end{prop}

P-matroids were extensively studied by Todd~\cite{t-com-84}. He lists
eight equivalent conditions for a matroid to be a P-matroid. We recall
three of them (conditions (a), (a*) and (c) below)
and add two new ones.

\begin{thm} \label{thm:eqP}
For an oriented matroid $\mathcal{M}$ on $E_{2n}$, the following
conditions are equivalent:

\begin{compactenum}[\rm(a)]
\item $\M$ has no s.r.~circuit;
\item[\rm(a*)] $\M$ has no s.p.~cocircuit;
\item every t.s.p.~$X$ is a vector of~$\M$;
\item[\rm(b*)] every t.s.r.~$Y$ is a covector of $\M$;
\item every one-point extension $\hat{\mathcal{M}}$ of $\mathcal{M}$
to~$\hat E_{2n}$ contains exactly one complementary circuit~$C$
such that $C\ge0$ and $C_q={+}$.
\end{compactenum}
\end{thm}

\begin{proof}
The equivalence of the conditions (a), (a*) and (c) was shown by
Todd~\cite{t-com-84}. Morris~\cite{Mor:Oriented} proved that (a)
implies~(b). We show the equivalence of (a) with~(b*). The
equivalence of (a*) with~(b) is proved analogously.

First we prove that (a) implies (b*). Since no circuit of~$\M$ is s.r.,
there is for every circuit~$C$ an element~$e$ such that $C_{e} \cdot
C_{\ole}={+}$. It follows that every t.s.r.\ sign vector~$Y$ is orthogonal
to every circuit, hence $Y$ is a covector.

For the opposite direction, suppose that there is a s.r.~circuit~$C$.
If so, then any t.s.r.~vector~$Y$ for which $Y^{+} \subseteq C^{+}$
and $Y^{-} \subseteq C^{-}$ is \emph{not} orthogonal to~$C$, which is
a contradiction with~(b*).
\end{proof}

The condition~(b) of this theorem has a translation for realization
matrices of P-ma\-troids, that is, for P-matrices:

\begin{cor}
A matrix $M\in\R^{n\times n}$ is a P-matrix if and only if for every
$\sigma\in\{-1,+1\}^n$ there exists a vector~$x\in\R^n$ such that
for $y=Mx$ and for each $i\in\{1,2,\dotsc,n\}$ we have
\begin{align*}
\sigma_i x_i &> 0,\\
\sigma_i y_i &> 0.
\end{align*}
\end{cor}

Todd~\cite{t-com-84} also gives an oriented-matroid analogue of the
``positive principal minors'' condition. Stating it would require some
more explanations; later in this article we need a weaker property of
P-matroids, though, which corresponds to the fact that all principal
minors of a P-matrix are nonzero.

\begin{lem}[Todd~\cite{t-com-84}] \label{lem:basisP}
 For a P-matroid $\mathcal{M}$ every complementary subset $B \subseteq
 E_{2n}$ of cardinality $n$ is a basis.
\end{lem}

\begin{remark}
In addition, every such complementary $B$ is also a cobasis, i.e.,
it is a basis of the dual matroid $\mathcal{M}^{*}$.
\end{remark}

Next we consider principal pivot transforms
(see~\cite{tsatsomeros2,Tuc:A-combinatorial}) of P-matrices.  The fact
that every principal pivot transform of a P-matrix is again a
P-matrix~\cite{Tuc:Principal} is well-known. The proof is not very
difficult but it uses involved properties of the Schur complement. In
the setting of oriented matroids the equivalent is much simpler. First
let us define principal pivot transforms of oriented matroids.

\begin{defn}
Let $F\subseteq E_{2n}$ be a complementary set. The \emph{principal
pivot transform} of a sign vector~$X$ with respect to~$F$ is the sign
vector~$\tilde C$ given by
\[\tilde C_e = \begin{cases}
	C_e&\text{if }e\notin F,\\
	C_{\ole}&\text{if }e\in F.
\end{cases}\]
The \emph{principal pivot transform} of a matroid~$\M$ with respect
to~$F$ is the matroid whose circuits (vectors) are the principal pivot
transforms of the circuits (vectors) of~$\M$.
\end{defn}

It is easy to see that, in the LCP-realizable case, principal pivot transforms
of a matroid correspond to matroids realized by corresponding principal
pivot transforms of the realization matrix. Thus the following proposition
implies the analogous theorem for P-matrices.

\begin{prop}
Every principal pivot transform of a P-matroid is a P-matroid.
\end{prop}

\begin{proof}
The principal pivot transform of a circuit~$C$ is \sr\ if and only if $C$~is \sr.
\end{proof}

\section{Z-matroids}

The second class of matrices we examine are Z-matrices; the corresponding
matroid generalizations are Z-matroids. Recall that a \emph{Z-matrix}
is a matrix whose every off-diagonal element is non-positive.
The definition of Z-matroids was first proposed in~\cite{LemLut:Classes}.

\begin{defn} 
 A matroid $\mathcal{M}$ on $E_{2n}$ is a \emph{Z-matroid} if for every
 circuit $C$ of~$\M$ we have:
 \begin{equation}
  \label{defZ}
  \begin{split}
  \text{If } C_{T} \geq 0 &\text{, then} \\
  & C_{\ole}={+} \text{ for all } e \in S \text{ with } C_{e}=+ .
  \end{split}
 \end{equation}
\end{defn}

\begin{remark}
In the definition of Z-matroid we could replace all occurrences of
the word ``circuit'' with the word ``vector''. Indeed, in a conformal
decomposition of a vector violating~\eqref{defZ}, there would always be
a circuit violating~\eqref{defZ} as well.
\end{remark}

It makes perfect sense to define Z-matroids in this way. We show that
in LCP-realizable cases, any realization matrix $M$ is a Z-matrix.

\begin{prop} \label{prop:realZ}
~
\begin{compactenum}[\rm(i)]
\item If $\M$ is LCP-realizable and there exists a realization matrix~$M$
that is a Z-matrix, then $\M$~is a Z-matroid.
\item If $\mathcal{M}$ is an LCP-realizable Z-matroid, then
every realization matrix~$M$ is a Z-matrix.
\end{compactenum}
\end{prop}

\begin{proof}
We fix $E_{2n}=\{1,\dotsc,2n\}$ with $S=\{1,\dotsc,n\}$ where the complement of an $i \in S$ is the element $i+n$.

(i) Let $e_{i}$ denote the $i$th unit vector and $m_j$ the $j$th column
of the matrix~$M$. The sign pattern of the Z-matrix
$M$ implies that there is no linear combination of the form
\begin{align*}
 e_{i} + \sum_{\substack{j=1\\j\neq i}}^{n} x_{j}e_{j} - \sum_{j=n+1}^{2n} x_{j}m_{j}=0,
\end{align*}
where $x_{j} \geq 0$ for every   $j > n$  and  $x_{i+n}=0 $, because
the $i$th row of the left-hand side is strictly positive. Hence there is
no vector $X \in\V$ for which $X_{T} \geq 0$, $X_{i}=+$ but $X_{i+n}=0$
for some $i \in S$.

(ii) Proof by contradiction. Assume that for an LCP-realizable Z-matroid
$\mathcal{M}$ (where $S$ is a basis), there is a realization matrix~$M$
that is not a Z-matrix, that is, there is an off-diagonal $m_{ij}>0$.
If so, there is a vector $X$ with $X_{j+n}=+$ and $X_{T \backslash
\left\lbrace j+n \right\rbrace }=0$, but $X_{i}=+$. This $X$
violates the Z-matroid property~\eqref{defZ} since also $X_{i+n}=0$,
a contradiction. Thus no positive $m_{ij}$ can exist and $M$ has to be
a Z-matrix.
\end{proof}

Another option is to characterize a Z-matroid with respect to the dual
matroid $\mathcal{M}^{*}$.

\begin{prop} \label{defZd}
An oriented matroid $\mathcal{M}$ on~$E_{2n}$ is a Z-matroid if and only
if for every cocircuit $D$ of~$\M$ we have:
 \begin{equation}
  \label{defZdual}
  \begin{split}
  \text{If } D_{S} \leq 0 &\text{, then} \\
  & D_{\ole}={-} \text{ for all } e \in T \text{ with } D_{e}=+ .
  \end{split}
 \end{equation}
\end{prop}

\begin{proof}
First we prove the ``only if'' direction. Suppose that there is a
cocircuit~$D$ which does not satisfy~\eqref{defZdual}.  Accordingly $D_{S}
\leq 0$ and there is $e \in T$ such that $D_{e}=+$, but $D_{\ole}=0$. But
note that the fundamental circuit $C:=C(S,e)$ and $D$ are not orthogonal
because the Z-matroid property~\eqref{defZ} implies that $C_{S \backslash
\{e\}} \leq 0$. Hence no such $D$ can exist.

For the ``if'' direction suppose that there is a circuit~$C$ for which
$C_{T} \geq 0$ and $C_{e}=+$, but $C_{\ole}=0$ for some $e \in S$. This
circuit~$C$ and the fundamental cocircuit $D:=D(T,e)$ are not orthogonal
since by assumption \eqref{defZdual} holds for $D$ and of course $-D$,
hence $D_{T \backslash \{\ole \}} \geq 0$.
\end{proof}

In the proofs in the following section we often make use of fundamental circuits.
Here we observe that all fundamental circuits with respect to the
basis~$S$ follow the same sign pattern.

\begin{lem}
\label{lem:fundc}
Let $\M$ be a Z-matroid.  Let $e\in T$ and let $C=C(S,e)$ be the
fundamental circuit of~$e$ with respect to the basis~$S$. Then
\begin{align*}
C_e &= +, \\
C_{T \backslash \left\lbrace e \right\rbrace} &= 0, \\
C_{S \backslash \left\lbrace \ole \right\rbrace } &\leq 0.
\end{align*}
\end{lem}

\begin{proof}
The first and the second equality follow directly from the definition of
a fundamental circuit. Thus $C_{T} \geq 0$. Hence the third property
follows from the Z-matroid property~\eqref{defZ}.
\end{proof}

\section{K-matroids}

\begin{defn}
 A matroid $\mathcal{M}$ on $E_{2n}$ is a \emph{K-matroid} if it is a P-matroid and a Z-matroid.
\end{defn}

Combining Proposition~\ref{realP} and Proposition~\ref{prop:realZ}
we immediately get:

\begin{prop} \label{prop:realK}
~
\begin{compactenum}[\rm(i)]
\item If $\M$ is LCP-realizable and there exists a realization matrix~$M$
that is a K-matrix, then $\M$~is a K-matroid.
\item If $\mathcal{M}$ is an LCP-realizable K-matroid, then
every realization matrix~$M$ is a K-matrix.
\end{compactenum}
\end{prop}

An oriented matroid minor $\mathcal{M} \backslash F \slash \olF $ where
$F$~is a complementary subset of~$E_{2n}$ is called a \emph{principal minor}
of $\mathcal{M}$.

\begin{lem} \label{lem:K}
Let $\mathcal{M}$ be a K-matroid. Then every principal minor
$\mathcal{M} \backslash F \slash \olF $ is a K-matroid.
\end{lem}

\begin{proof}
It was shown by Todd~\cite{t-com-84} that every principal minor of a
P-matroid is a P-matroid. Thus, it is enough to show that such a minor
is a Z-matroid, and for this, since deletions and contractions can be 
carried out element by element in any order, it suffices to consider
the case that $F$~is a singleton.

First, we prove that if $e \in T$, then $\mathcal{M} \backslash
\left\lbrace e\right\rbrace  \slash \left\lbrace \ole \right\rbrace$
is a Z-matroid. Such a principal minor consists of all circuits $C
\backslash \left\lbrace e, \ole \right\rbrace$, where $C$~is a circuit
of~$\M$ and $C_{e}=0$. Since every circuit of~$\mathcal{M}$ satisfies the
Z-matroid characterization $\eqref{defZ}$, such a circuit $C \backslash
\left\lbrace e, \ole \right\rbrace$ trivially satisfies it too.

Secondly, let $e \in S$. Here we apply a case distinction. If
$C_{\ole}=+$, then $(C \backslash \left\lbrace e, \ole \right\rbrace)_{T}
\geq 0$ if and only if $C_{T} \geq 0$. As a direct consequence, $C
\backslash \left\lbrace e, \ole \right\rbrace$ satisfies~\eqref{defZ}
because $C$ does. If $C_{\ole}=-$, we can show that there is another
element $f \in T$  such that $C_{f}=-$ too, that is, $(C \backslash
\left\lbrace e, \ole \right\rbrace)_{T} \not \geq 0$ and thus the
Z-matroid property $\eqref{defZ}$ is obviously satisfied. Assume for the
sake of contradiction that there is no such $f \in T$. The strong circuit
elimination~\ciref{c4p} of $C$ and the fundamental circuit $C(S,\ole)$
at $\ole$ then yields a circuit $C'$ with $C'_{T} \geq 0$, $C'_{\ole}=0$
and $C'_{e}=+$. Since $e \in S$, such a $C'$ would violate the Z-matroid
definition, a contradiction.
\end{proof}

Our main result, the combinatorial generalization of the Fiedler--Pt\'ak
Theorem~\ref{thm:FP} is the following.

\begin{thm} \label{thm:eqK}
For a Z-matroid $\mathcal{M}$ (with vectors~$\V$, covectors~$\V^{*}$,
circuits~$\C$ and cocircuits~$\D$), the following statements are
equivalent:

\begin{tabular}{l c l}
{\rm(a)}  $\exists X \in \mathcal{V}: X_{T} \geq 0 \text{ and } X_{S} > 0$; & &
{\rm(a*)} $\exists Y \in \mathcal{V}^{*}: Y_{S} \leq 0 \text{ and } Y_{T} > 0$;\\
{\rm(b)}  $\exists X \in \mathcal{V}: X > 0$; & &
{\rm(b*)} $\exists Y \in \mathcal{V}^{*}: Y_{S} < 0 \text{ and } Y_{T} > 0$;\\
{\rm(c)}  $\forall C \in \mathcal{C}$: $C_{S} \geq 0 \implies C_{T} \geq 0$; & &
{\rm(c*)} $\forall D \in \mathcal{D}$: $D_{T} \geq 0 \implies D_{S} \leq 0$; \\
{\rm(d)}  there is no s.r.\ circuit $C \in \mathcal{C}$ &&
{\rm(d*)} there is no s.p.\ cocircuit $D \in \mathcal{D}$. \\
\phantom{(d)} (that is, $\M$ is a P-matroid);
\end{tabular} 
\end{thm}

In order to use duality in the proof of this theorem, let us
first define the \emph{reflection} of a matroid $\M=(E_{2n},\V)$
to be the matroid $\reflection(\M)=(E_{2n},\reflection(\V))$, where
$\reflection(\V)=\{\reflection (X): X\in\V\}$ with
\[
\bigl(\reflection (X)\bigr)_e = 
\begin{cases}
X_{\ole} & \text{if $e\in S$},\\
-X_{\ole} & \text{if $e\in T$}.
\end{cases}
\]
Observe that $\reflection\bigl(\reflection(\M)\bigr)=\M$ because
of~\veref{v2}, and that $\reflection(\M^\ast)=\reflection(\M)^\ast$;
thus
\begin{equation}
\label{eq:reflstar}
\reflection\bigl(\reflection(\M^\ast)^\ast\bigr)=\M.
\end{equation}

\begin{proof}[Proof of Theorem~\ref{thm:eqK}]~

\begin{description}
\item{(a)${}\implies{}$(b):}
  Let $X$ be as in (a). Since $X_{T} \geq 0$, the Z-matroid
  property~\eqref{defZ} implies that if $X_{e}=+$ for an $e \in S$,
  then $X_{\ole}=+$.  Thus $X_{T} > 0$.

\item{(b)${}\implies{}$(c):}
  Let $X$ be the all-plus vector as in~(b).  Suppose that there is
  a circuit $C \in \mathcal{C}$ not satisfying~(c), that is, $C_{S}
  \geq 0$ but $C_{e}=-$ for some element~$e$ in~$T$. Starting with
  $Y^{0}=C$, we apply a sequence of vector eliminations~\veref{v4} to
  get vectors $Y^{i}$. We eliminate $Y^{i-1}$ and $X$ at any element $e$
  where $Y^{i-1}_{e}=-$. For a resulting vector $Y^{i}$ it holds that
  $(Y^{i})^{-} \subset (Y^{i-1})^{-}$. Thus, at some point $(Y^{k})^{-}=
  \emptyset$ while $Y^{k}_{e}=0$ and $Y^{k}_{\ole}=+$ where $e \in T$
  is the element eliminated in step $k-1$. This vector $Y^{k}$ does not
  satisfy the Z-matroid property~\eqref{defZ}, which is a contradiction.

\item{(c)${}\implies{}$(d):}
  Suppose that there is a s.r.~circuit $C \in \mathcal{C}$, that is,
  $C_{e} \cdot C_{\ole} \leq 0$ for every $e \in S$. Let $C^{0}=C$. We
  apply consecutive circuit eliminations~\ciref{c4}. To get $C^{i}$,
  we eliminate $C^{i-1}$ with any fundamental circuit $\check C:=C(S,e)$
  at position $e \in T$ where $C^{i-1}_{e}=-$. By Lemma~\ref{lem:fundc}
  we have $\check C_{S\setminus\{\ole\}}\le 0$.

  After finitely many eliminations we end up with a circuit $C^{k}$ for
  which $C^{k}_{T} \geq 0$. Now we claim that $C^k_S\le0$: Indeed, if
  $e\in S$ such that $C^k_{\ole} = {+}$, then $C_{\ole} = +$, and thus
  $C_e\le 0$ because $C$~is \sr. Since we never eliminate at~$\ole$,
  all fundamental circuits~$\check C$ used in the eliminations satisfy
  $\check C_e\le 0$ as noted above. Hence $C^k_e\le0$. If on the other
  hand $C^k_{\ole} = 0$, then $C^k_e\le0$ by~\eqref{defZ}.

  Moreover, since $S$~is a basis, $\underline{C^k}\nsubseteq S$, and so
  there exists $e\in T$ with $C^k_e={+}$. Therefore $-C^{k}$ violates
  property~(c), a contradiction.

\item{(d)${}\implies{}$(a*):}
  Because of~(d), for every circuit $C$ there is an $e \in S$ such
  that $C_{e} \cdot C_{\ole}=+$. The sign vector $Y$ where $Y_{S} < 0$
  and $Y_{T} > 0$ is orthogonal to every circuit $C$, because the sign
  of $Y_{e} \cdot C_{e}$ is opposite to the sign of $Y_{\ole} \cdot
  C_{\ole}$. Hence such a $Y$ is a covector.
\end{description}

To finish the proof, notice that a matroid~$\M$ satisfies~(a*) if and
only if the reflection of its dual~$\reflection(\M^\ast)$ satisfies~(a);
analogously for (b*) and~(b), (c*) and~(c), and (d*) and~(d). Thus if
$\M$ satisfies~(a*), then $\reflection(\M^\ast)$~satisfies~(a), hence
also~(b), and so (using~\eqref{eq:reflstar}) $\M$~satisfies~(b*). The
missing implications (b*)${}\implies{}$(c*), (c*)${}\implies{}$(d*),
and (d*)${}\implies{}$(a) are proved analogously.
\end{proof}


\section{Algorithmic aspects} \label{sec:algo}

Let an OMCP($\hat{\mathcal{M}}$)  be given, where $\hat{\mathcal{M}}$~is
any one-element extension of an $n$-dimensional matroid~$\mathcal{M}$
on $E_{2n}$. We present \emph{simple principal pivot algorithms} to find
a solution. This kind of algorithm is a well-established solving method
for LCPs. Sometimes called \emph{Bard-type methods}, they were first
studied by Zoutendijk~\cite{Zou:Methods} and Bard~\cite{Bar:An-eclectic}.

Here we extend a recent result of~\cite{FonFukGar:Pivoting} to the
generalizing setting of OMCP. We prove below that the unique solution to
every OMCP($\hat{\mathcal{M}}$) where the underlying matroid $\mathcal{M}$
is a K-matroid, is found by every \spp~algorithm in a linear number of
pivot steps.

Let $\hat{\mathcal{M}}$ be given by an oracle which, for a basis~$B$ of
$\hat{\mathcal{M}}$  and a non-basic element $e \in \hat E_{2n}  \backslash B$,
outputs the unique fundamental circuit $C(B,e)$. A \spp~algorithm starts with a
fundamental circuit $C^{0}=C(B^{0},q)$ where $B^{0}$ is any
complementary basis. For instance, we can take $B^{0}=S$. It then proceeds in \emph{pivot
steps}. Assume that the $i$th step leads to a fundamental circuit
$C^{i}=C(B^{i},q)$. We require the complementary
condition~\eqref{V2COMP} to be an invariant, that is, $B^{i}$~is supposed to be
complementary. If $C^{i}$~is feasible, that is, the condition~\eqref{VPOS}
is satisfied, then $C^{i}$~is the solution and the algorithm
terminates. Otherwise, we obtain~$C^{i+1}$ as follows: choose an $e^{i}
\in B^{i}$ for which $C^{i}_{e^{i}}=-$ according to a \emph{pivot
rule}. Then the pivot
element~$e^{i}$ is replaced in the basis with its complement~$\overline{e^{i}}$,
that is, $B^{i+1}=B^{i}\backslash \{e^{i}\} \cup \{\overline{e^{i}}\}$.
Lemma~\ref{lem:basisP}
asserts that $B^{i+1}$ is indeed a basis. Then
$C^{i+1}=C(B^{i+1},q)$ is computed by feeding the oracle with basis
$B^{i+1}$ and the non-basic element~$q$. The algorithm then
proceeds with pivot step $i+2$.

\bigskip

\vbox{
\textsc{SimplePrincipalPivot}($\hat{\mathcal{M}}, B^{0}$)
\begin{algorithmic}
\STATE $i:=0$
\STATE $C^{0}:=C(B^{0},q)$
\WHILE{$(C^{i})^{-} \neq \emptyset$}
\STATE $\text{choose } e^{i} \in (C^{i})^{-} \text{ according to a pivot rule R}$
\STATE $B^{i+1}:=B^{i} \backslash \bigl\lbrace e^{i}
\bigr\rbrace \cup \bigl\lbrace \overline{e^{i}} \bigr\rbrace $
\STATE $C^{i+1}:=C(B^{i+1},q)$
\STATE $i:=i+1$
\ENDWHILE
\RETURN $C^{i}$
\end{algorithmic}
}

\bigskip

If the number of pivots is polynomial in $n$, then the whole algorithm
runs in polynomial time too, provided that the oracle computes the
fundamental circuit in polynomial time. This is the case if the LCP is
given by a matrix~$M$ and a right-hand side~$q$ as in~\eqref{eq:LCP}.

The number of pivots depends on the applied pivot rule and some rules
may even enter a loop on some inputs $\hat{\mathcal{M}}$. If the input
is a K-matroid extension, though, then the \textsc{SimplePrincipalPivot}
method is fast. We claim that no matter which pivot rule is applied,
\textsc{SimplePrincipalPivot} runs in a linear number of pivot steps
on every K-matroid extension. The following two lemmas are required to
prove this fact.  While the first one holds for every P-matroid extension,
the second is restricted to K-matroid extensions.

\begin{lem} \label{lem:P}
 If $\hat{\mathcal{M}}$ is a P-matroid extension, then
 $C^{i+1}_{\overline{e^{i}}} = {+}$ for every $i \geq 0$.
\end{lem}

\begin{proof}
 First suppose that $C^{i+1}_{\overline{e^{i}}}=-$ in some pivot step~$i+1$. Let $C'$
 be the result of a weak circuit elimination of $C^{i}$ and $-C^{i+1}$
 at~$q$. Then $C'$ is contained in~$B^i\cup\{\overline{e^i}\}$,
 and $C'_{e^{i}}\le0$ and $C'_{\overline{e^{i}}}\ge0$, in other
 words it is a s.r.~vector. According to the Definition~\ref{def:P}
 of a P-matroid, no s.r.~circuit can exist. Thus $C^{i+1}_{\overline{e^{i}}}\ge0$.

 Now suppose that $C^{i+1}_{\overline{e^{i}}}=0$. Then $C^{i+1}$ is
 also the fundamental circuit $C(B^i,q)$, hence $C^{i+1}=C^i$. This is
 a contradiction because $C^{i+1}_{e^i}=0\ne {-} = C^i_{e^i}$.
\end{proof}

\begin{lem} \label{lem:stays+}
 If $\hat{\mathcal{M}}$ is a K-matroid extension, then for every $f \in T$:
 $$\text{If } C^{h}_{f}\ge0 \text{ for some } h \geq 0,
 \text{ then } C^{k}_{f}\ge0 \text{ for every } k \geq h.$$
\end{lem}

\begin{proof}
For the sake of contradiction suppose that the statement does not hold
and let $l \geq h$ be the smallest value such that $C^{l}_{f} \geq 0$,
but $C^{l+1}_{f}=-$. By Lemma~\ref{lem:P}, $f\ne\overline{e^l}$, and so
$f$~lies in~$B^l$ and in~$B^{l+1}$. Let $X$ be the result of a vector
elimination of
$C^{l}$ and $-C^{l+1}$ at~$q$. Note that $X_{e^{l}}=-$, $X_{f}=+$ and
$X_{\olf }=0$. In addition by
Lemma~\ref{lem:P} it holds that $X_{\overline{e^{l}}} = {-}$. Since $X_{q}=0$,
the sign vector
$X \backslash \{q\}$ is a vector of the K-matroid~$\mathcal{M}$. Now
let $F:=\overline{B^l}\setminus\{\olf,\overline{e^l}\}$. Consider the principal
minor $\mathcal{M} \backslash F / \overline{F}$, which is a matroid on
the element set $\bigl\{ f, \olf , e^{l}, \overline{e^{l}} \bigr\}$. By
Lemma~\ref{lem:K} it is also a K-matroid. Further it contains the vector
$X'=X \backslash \Bigl(\hat E_{2n} \backslash \bigl\{f, \olf , e^{l},
\overline{e^{l}}\bigr\}\Bigr)$
with $X'_{e^{l}}=-$,  $X'_{f}=+$, $X'_{\olf }=0$ and
$X'_{\overline{e^{l}}} = {-}$. The contradiction follows from
the fact that $-X'$ violates the K-matroid property~(c) in
Theorem~\ref{thm:eqK}.
\end{proof}

\begin {thm} \label{thm:fastK}
Every simple principal pivot algorithm runs in at most $2n$ pivot steps on every K-matroid extension.
\end{thm}
\begin{proof}
 We prove that, no matter which pivot rule R one applies, every element
 $e \in E_{2n}$ is chosen at most once as the pivot element. Consider any pivot
step $h$ in the \textsc{SimplePrincipalPivot} algorithm. First suppose that
the pivot element $e^{h}$ is in $S$. According to Lemma~\ref{lem:P}
$C^{h+1}_{\overline{e^{h}}} \geq 0$. Moreover, for every $k \geq h$ we
have $C^{k}_{\overline{e^{h}}} \geq 0$ (Lemma~\ref{lem:stays+}) and
$C^{k}_{e^h}=0$. In
other words, the elements $e^{h}$ and $\overline{e^{h}}$ cannot become pivot
elements in
later steps. Second, suppose that the pivot $e^{h}$ is in $T$. Then the
argument above fails. Even though $C^{h+1}_{\overline{e^{h}}} \geq 0$
(Lemma~\ref{lem:P}), we cannot conclude that $C^{k}_{\overline{e^{h}}} \geq
0$ for every $k \geq h$, because Lemma~\ref{lem:stays+} does not apply. It
may eventually happen
for some $k$ that $\overline{e^{h}}$ is chosen as pivot $e^{k}$.
However if so, our first argument applies for pivot step $k$ and neither
$\overline{e^{h}}$ nor $e^{h}$ can become
pivot elements again.
\end{proof}

\begin{remark}
If \textsc{SimplePrincipalPivot} starts with the basis $B^0=S$, then
at most $n$ pivot steps are needed, because $C^0_T=0$ and hence, by
Lemma~\ref{lem:stays+}, $C^i_T\ge0$ for all~$i$.
\end{remark}

\section{Extension to principal pivot closures}

So far, we have considered a matroid $\mathcal{M}$ on a complementary set
$E_{2n}$ where the maximal complementary set~$S$ is fixed from the
beginning. In the following, $S'$~is an arbitrary complementary subset
of size~$n$ and $T'=\{\ole :e \in S'\}$.

\begin{defn}
 A matroid $\mathcal{M}$ on $E_{2n}$ is a \emph{Z*-matroid} if there is a
 complementary set $S' \subseteq E_{2n}$ of cardinality~$n$ such that for
 $T'=\{\ole:e\in S'\}$ and every circuit~$C$ of~$\M$ we have:
 \[
  \begin{split}
  \text{If } C_{T'} \geq 0 &\text{, then} \\
  & C_{\ole}={+} \text{ for all } e \in S' \text{ with } C_{e}=+ \text .
  \end{split}
 \]
\end{defn}

Analogously $\mathcal{M}$ is a \emph{K*-matroid} if it is a P-matroid and
a Z*-matroid. Note that the class of Z*- and K*-matroids are the closures,
under all principal pivot transforms, of Z- and K-matroids,
respectively.  Moreover, Proposition~\ref{defZd}, Lemma~\ref{lem:K}
up to Theorem~\ref{thm:fastK} have equivalent counterparts for these
closure classes, obtained by substituting $S$ by $S'$ and accordingly $T$
by $T'$ in the original statements. Hence we get the following.

\begin{cor}
 Every simple principal pivot algorithm finds the solution to
 OMCP($\mathcal{\hat{M}}$), where $\mathcal{\hat{M}}$ is a K*-matroid
 extension, in at most $2n$ pivot steps.
\end{cor}

The reader might wonder why we introduced Z-matroids and K-matroids at
all and did not start off with their principal pivot closures. One good
reason for our approach is to point out the correspondence
of LCP-realizable Z-matroids and their matrix counterparts, see
Proposition~\ref{prop:realZ}. With respect to this, the main problem is that
a principal pivot transform of a Z-matroid or a K-matroid is in general
not a Z-matroid, a K-matroid respectively. However every principal pivot
algorithm is still able to solve an LCP($M,q$) where $M$~is a principal pivot
transform of a K-matrix in a linear number of pivot steps.

\end{document}